\documentclass[a4paper,10pt]{amsart}

\usepackage[latin1]{inputenc}
\usepackage[T1]{fontenc}
\usepackage[english]{babel}

\usepackage{amsmath,amssymb,amsthm,enumerate,xspace}

\usepackage{mathrsfs}
\usepackage{amscd}
\usepackage{amsfonts}
\usepackage{amsmath}
\usepackage{amssymb}
\usepackage{amstext}
\usepackage{amsthm}
\usepackage{amsbsy}

\usepackage{xspace}
\usepackage[all]{xy}
\usepackage{graphicx}
\usepackage{url}
\usepackage{latexsym}

\usepackage{hyperref}

\usepackage{cleveref}
\usepackage{tikz-cd}

\newtheorem{theo}{Theorem}[section]
\newtheorem{lem}[theo]{Lemma}
\newtheorem{prop}[theo]{Proposition}
\newtheorem{cor}[theo]{Corollary}

\newtheorem{que}[theo]{Question}
\newtheorem{claim}[theo]{Claim}

\newtheorem{thmintro}{Theorem}

\newtheorem{corintro}[thmintro]{Corollary}

\newtheorem{innercustomgeneric}{\customgenericname}
\providecommand{\customgenericname}{}
\newcommand{\newcustomtheorem}[2]{%
  \newenvironment{#1}[1]
  {%
   \renewcommand\customgenericname{#2}%
   \renewcommand\theinnercustomgeneric{##1}%
   \innercustomgeneric
  }
  {\endinnercustomgeneric}
}

\newcustomtheorem{customthm}{Theorem}
\newcustomtheorem{customprop}{Proposition}
\newcustomtheorem{customlemma}{Lemma}
\newcustomtheorem{customcorollary}{Corollary}

 \theoremstyle{definition}
\newtheorem{dfn}[theo]{Definition}

\newtheorem{ex}[theo]{Example}
\newtheorem{exs}[theo]{Examples}

\newtheorem{fact}[theo]{Fact}
\newtheorem{rem}[theo]{Remark}
\newcommand{\N}{\ensuremath{\mathbb{N}}}  
\newcommand{\Z}{\ensuremath{\mathbb{Z}}}
\newcommand{\Q}{\ensuremath{\mathbb{Q}}}
\newcommand{\R}{\ensuremath{\mathbb{R}}}

\newcommand{\M}{\ensuremath{\mathcal{M}}}
\newcommand{\Rs}{\ensuremath{\mathcal{R}}}

\newcommand{\vs}{\vspace{0.2cm}}

\newcommand{\mtf}{\ensuremath{\mathcal{N}}}

\newcommand{\se}{\subseteq}
\newcommand{\inv}{^{-1}}

\newcommand{\ol}{\overline}

\newcommand{\df}{\rangle_\mathrm{def}}
\newcommand{\la}{\langle}
\newcommand{\ra}{\rangle}

\newcommand{\fg}{\widehat G} 
\DeclareMathOperator{\Aut}{Aut}

\DeclareMathOperator{\GL}{GL} 
 
\DeclareMathOperator{\SO}{SO}

\DeclareMathOperator{\rk}{rk_\mathrm{def}}

\renewcommand{\leq}{\leqslant}
\renewcommand{\geq}{\geqslant}

\begin{document}

\author{Annalisa Conversano}

\title[Definable rank, o-minimal groups, and Wiegold's problem]{Definable rank, o-minimal groups,  \\ and Wiegold's problem}

\address{Massey University Auckland, New Zealand} 
\email{a.conversano@massey.ac.nz}

\noindent
\date{June 24, 2026} 
 
\maketitle

\begin{abstract} 
We show that an o-minimal structure $\M$ defines groups with infinite definable rank if and only if \M\ defines some finite power of $\Q/\Z$. If no interval of \M\ is countable, then all groups definable in \M\ have finite definable rank.

In general, we prove that every definable group $G$ in an arbitrary o-minimal structure is an extension of a definable periodic group $P$ by a (maximal unique) definably connected definably finitely generated subgroup $\fg$. When $G$ is definably connected, $P$ is abelian and the extension almost split, in that $G$ is an almost direct product 
$G = (\fg \times P)/F$, for some finite central subgroup $F$. The definable rank of $\fg$ is bounded above by its dimension, and the upper bound is strict whenever $\fg$ is not solvable.

Along the way, we show that every linear definable group has finite definable rank. This provides another proof, and a generalization to linear o-minimal groups, of the fact that linear algebraic groups over an algebraically closed field of characteristic $0$ contain a Zariski-dense finitely generated subgroup.

We further prove that every perfect definable group is normally monogenic, generalizing the finite group case. This yields a positive answer to Wiegold's problem in the o-minimal setting.

\end{abstract}

\thispagestyle{empty}

\bigskip
\section{Introduction}

Every algebraic group over an algebraically closed field $\mathcal{K} = (K, +, \cdot)$ is definable in $\mathcal{K}$. In characteristic $0$, the field $\mathcal{K}$ is definable in any maximal real closed field $\Rs$. It follows that every algebraic group $G$ over $\mathcal{K}$ is definable in the o-minimal structure $\Rs$. On the other hand, a considerable volume of work shows that definable groups in arbitrary o-minimal structures are very closely related to algebraic groups in general. For the compact case, see for instance \cite{BB12, Berarducci-Mamino, BMO10, EMPRT, HPP1, PC}. For the non-compact case, see \cite{BBO19, me-nilpotent, JC, PPSI, PPSII, PPSIII}.

In this paper we present yet again another common feature by showing, among other things, that every linear definable group in an arbitrary o-minimal structure is definably finitely generated, generalizing the well-known fact that linear algebraic groups over an algebraically closed field of characteristic $0$ contain a finitely generated Zariski-dense subgroup \cite[Proposition 1]{Tret}. 

\begin{thmintro}\label{theo:linear}
 Let $G$ be a definable group in an o-minimal structure \M. If $G$ is definably isomorphic to a definable subgroup of $\GL_m(\Rs)$ for some positive integer $m$ and definable real closed field $\Rs$, then $G$ is definably finitely generated.  
\end{thmintro}

Let us now fix the terminology and notation. 
The rank of a group is the smallest cardinality of a generating set.
By analogy, we call the {\bf definable rank} of a definable group $G$, the smallest cardinality of a set $X$ such that $G = \la X \df$, where $\la X \df$ denotes the smallest definable subgroup of $G$ containing $X$. We will say that $X$ is a {\bf definably generating set} for $G$. In o-minimal structures, $\la X \df$ always exists because the descending chain condition for definable groups holds \cite[Proposition 2.12]{Pillay - groups}.

If $G$ has finite definable rank $n \in \N$, we say that $G$ is {\bf definably finitely generated}, and write $\rk(G) = n$. Following \cite{Strzebonski}, when  $\rk(G) \leq 1$ we say that $G$ is {\bf monogenic}.  That is, $G$
is monogenic when $G = \la g \df$, for some $g \in G$. Every monogenic definable group is abelian, and so is $\la X \df$ whenever the elements of $X$ commute pairwise. Disconnected monogenic definable groups are described in \Cref{prop:ab-conn} and \Cref{cor:monogenic}. 

For a definable group, a definably generating set is quite significant, as it plays the same role as a generating set of an abstract group. For instance, every definable homomorphism is uniquely determined by the value on a definably generating set.  Not  every definable group in an o-minimal structure is definably finitely generated. For instance,  in the o-minimal $\M = (\Q, <, +)$ the definable group 
$A = [0, 1[$ with addition modulo 1 is an infinite periodic group isomorphic to $\Q/\Z$, and every finitely generated subgroup is finite. Hence $A$ has infinite definable rank.

We will show that in some sense this is the \emph{only} definable group with infinite definable rank in any o-minimal structure, in that a definable group $G$ has infinite definable rank if and only if $G$ is a definable extension of a finite power of $\Q/\Z$. More precisely:

\begin{thmintro}\label{theo:def-fg}
	In every definable group $G$ there is a definably connected subgroup $\fg$ with finite definable rank containing all definably connected subgroups of $G$ with finite definable rank. The subgroup $\fg$ is definably characteristic, the quotient $G/\fg$ is a periodic group and $G^0$ is an almost direct product of $\fg$ and a subgroup isomorphic to $G^0/\fg$. Namely,
	\[
	G^0 \cong (\fg \times  (\Q/\Z)^p)/F,
	\]
	
	\medskip \noindent
	where $F$ is a finite subgroup central in $G^0$ and $p = \dim (G/\fg)$. The group $G$ is definably finitely generated if and only if $G^0 = \fg$. Moreover,    
	\[
	\rk(\fg) \leq \dim \fg \qquad \& \qquad \rk(\fg) = \dim \fg\ \Rightarrow\ \fg \mbox{ solvable.}
	\]
\end{thmintro}
 
 \begin{corintro} \label{cor:countable}
The definable rank of a definable group is at most countable.
 \end{corintro}

 It follows from \Cref{theo:def-fg} that an o-minimal structure \M\ defines groups with infinite definable rank if and only if \M\ defines a finite power of $\Q/\Z$. We will show this can happen only when \M\ has countable intervals:
 
 \begin{thmintro}\label{theo:str}
   If no interval of \M\ is countable then every $0$-group is monogenic, every definable periodic group is finite, and all groups definable in \M\ are definably finitely generated.
 \end{thmintro}

Both \Cref{theo:def-fg} and \Cref{theo:str} assume a full characterization of infinite periodic definable groups that turn out to be always finite extensions of a finite power of $\Q/\Z$ 
(\Cref{prop:periodic}). It follows from \Cref{theo:str} that in general the definable rank is not preserved in elementary extensions, and every sufficiently saturated o-minimal structure defines only groups with finite definable rank. More precisely, if $G$ is a definably connected group in \M\ with infinite definable rank and $\mathcal{N}$ is a $\aleph_1$-saturated elementary extension of \M, then 
$G(\mathcal{N})/\fg(\mathcal{N})$ is monogenic and $G(\mathcal{N})$ has finite definable rank. This begs the question of which   model-theoretic properties on a structure, outside the o-minimal case, may ensure finite definable generation for all definable groups.

\medskip
Recall that a group $G$ is said to be {\bf normally generated} by $x \in G$, if $x$ is not contained in any proper normal subgroup of $G$. In general, the smallest normal subgroup of $G$ containing $x$ is called the {\bf normal closure} of $x$ and denoted by $\la x \ra^G$. When $G$ is finitely generated, $G$ is the normal closure of a single element if and only if $G$ is a homomorphic image of a knot group \cite[Theorem 1]{knot}.

Every perfect finite group is normally generated by a single element (see for instance \cite[4.2]{LW}). For infinite groups, it is a well-known open problem since the 1970s attributed to J. Wiegold (see \cite[FP14]{problems02} and \cite[5.52]{notebook}): 

\begin{que}[Wiegold] 
Suppose $G$ is a finitely generated perfect group. Is $G$ normally generated by a single element? 
\end{que}

If $G$ is not assumed to be finitely generated, infinite direct sums of perfect groups provide easy counterexamples. In a first-order setting, we want to restrict ourselves to definable subgroups.

 \begin{dfn}
 We say that a definable group $G$ is {\bf normally monogenic} when there is some $x \in G$ that is not contained in any proper normal definable subgroup of $G$. We call such an element $x$ a {\bf normal generator} of $G$ and write $\la x \df^G = G$.
 \end{dfn}

Although it is commonly believed that finitely generated perfect counterexamples do exist, a positive answer to Wiegold's question has been given in \cite{EM13} for the compactly generated locally compact groups that do not admit infinite discrete quotients. The compactly generated assumption is the necessary replacement of finite generation. In this paper we give a positive solution in the o-minimal setting, with no further assumption:

\begin{thmintro}\label{theo:Wiegold}
Every perfect definable group is normally monogenic.
\end{thmintro}

The proof of \Cref{theo:Wiegold} is constructive and a normal generator is found explicitly. Two main structural ingredients are exploited: 1) the decomposition of centerless semisimple definable groups into a direct product of definably simple groups \cite[Theo 4.1]{PPSI}, each containing a definable torus \cite[Theo 5.1]{PPSIII} (for the connected case), and 2) the existence of a smallest definable subgroup of finite index \cite[Prop 2.12]{Pillay - groups}, (for the disconnected case).

\medskip
Unless otherwise stated, throughout the paper groups are definable in an arbitrary o-minimal structure \M. By {\bf definable}, we always mean ``definable with parameters in \M''. When we say that a set is {\bf definably connected} or {\bf definably compact}, we assume it is definable.
%

\medskip
\section{Linear groups have finite definable rank}

Recall that a definable group $G$ is called {\bf linear} when there is a definable real closed field $\Rs$ such that $G$ is definably isomorphic to a definable subgroup of the general linear group 
$\GL_m(\Rs)$, for some positive integer $m$. In this section we will show that every definable linear group has finite definable rank, generalizing Proposition 1 in \cite{Tret} showing that linear algebraic groups over an algebraically closed field of characteristic $0$ contain a Zariski-dense finitely generated subgroup.

In abstract groups, the rank of a subgroup can be bigger than the rank of the group. For instance, the free group with $n+1$ generators $F_{n+1}$ embeds in the free group with $n$ generators $F_n$, and the commutator subgroup of $F_2$ has infinite rank \cite{free}.
Similarly, there are definable groups containing definable subgroups with bigger definable rank. An example is given below.  

\begin{ex}\label{ex:subgr}
 Let $S$ be the semialgebraic group in \cite[5.3]{Strzebonski}. That is, $S = \R \times [1, e[$ with the operation defined by
 \[
 (x, u) \ast (y, v) = 
 \begin{cases}
 (x + y, uv)  & \mbox{ if $uv < e$} \\
 (x+y+1, uv/e) & \mbox{ otherwise.}
 \end{cases}
 \]

Fix $n \in \N$, $n > 1$, and set $G$ to be the direct product of $n$ copies of $S$. The subgroup $H = (\R \times \{1\})^n$ is definably isomorphic to $(\R^n, +)$, so $\rk(H) = n$. However, $G$ is monogenic. For instance, take $\alpha_1, \dots, \alpha_n \in (0, 1)$ such that $1, \alpha_1, \dots, \alpha_n$ are $\Q$-linearly independent, and set $u_i = e^{\alpha_i}$. Then $G = \la g \df$ for $g = ((0, u_1), \dots, (0, u_n))$.
\end{ex}

Definable quotients are better behaved, instead:

\begin{lem}\label{lem:rk-dec}
Let $H$ be a normal definable subgroup of a definable group $G$. Then
\[
\rk(G/H) \leq \rk(G) \leq \rk(H) + \rk(G/H). 
\]
\end{lem}

\begin{proof}
Let $X = \{g_i : i \in I\} \subset G$ be a set of minimal cardinality of definable generators and
$\ol X = \{\ol g_i : i \in I\} \subset G/H$ the set of their images by the canonical projection $G \to G/H$. If $\la \ol X \df$ is a proper subgroup of $G/H$, its pre-image is a proper definable subgroup of $G$ containing $X$, contradiction. Therefore  $\rk(G/H) \leq \rk(G)$.

Let $s \colon G/H \to G$ be a section of the canonical projection $G \to G/H$.  Suppose $Y \subset G/H$ is a subset of minimal cardinality such that $\la Y \df = G/H$ and set $Y' = \{s(y) \in G : y \in Y\}$. Then $G = \la H \cup Y' \df$ and 

\[
\rk(G) \leq \rk(H) + \rk(G/H),
\]

\noindent \medskip
as claimed.
\end{proof}

Every definable group $G$  contains a smallest definable subgroup of finite index $G^0$   \cite[Prop 2.12]{Pillay - groups}. $G^0$ is normal and coincides with the definably connected component of the identity. A definable group $G$ is {\bf definably connected} if and only if $G = G^0$. From \Cref{lem:rk-dec} we have:

\begin{cor}\label{cor:disc}
Let $G$ be a definable group. Then $\rk(G) \leq \rk(G^0) + |G/G^0|$. In particular, if $G^0$ has finite definable rank, so does $G$.
\end{cor}

The converse of \Cref{cor:disc} holds as well, in the sense that if $G$ has finite definable rank, so does $G^0$. In a previous version of the paper, we proved this fact by relying on further results. However, the referee rightly pointed out that this is unnecessary and suggested the following direct proof.

\begin{lem}\label{claim:gzero}
If $G$ has finite definable rank then $G^0$ has finite definable rank.	 
\end{lem} 

\begin{proof}
Let $X$ be a finite subset of $G$ such that $G = \la X \df$ and set $B = \la X \ra$. The subgroup $B \cap G^0$ has finite index in $B$, hence it is finitely generated. Let $Y$ be a finite 
set such that $\la Y \ra = B\cap G^0$ and set $H = \la Y \df$. Because $G^0$ is normal in $G$, $B \cap G^0$ is normal in $B$, $B$ normalizes $H$ and $BH$ is a subgroup of $G$. Moreover, $H$ has finite index in $BH$, so $BH$ is definable. Hence $BH = \la X \df = G$. It follows that $H$ has finite index in
$G$. Since $Y \subset G^0$, we get $G^0 = H = \la Y \df$.	 
\end{proof}

\medskip
Monogenic groups do not need to be definably connected. For instance:

\begin{ex} \label{ex:mono-disc} 
Let $G = \Z/2\Z \times \R$. For each $x \in \R$, $x \neq 0$, we can check that 
\[
G = \la (1, x) \df.
\]
 
 \noindent
Indeed, set $K = \la (1, x) \df$. Since $(1, x)^2 = (0, 2x) \in K$, $\la (0, 2x) \df = \{0\} \times \R \subset K$. In particular, $(0, x) \in K$ and $(1, x)(0, x)\inv = (1, 0) \in K$. 
\end{ex}

In fact, we can show that an abelian definable group $G$ is monogenic if and only if 
$G^0$ and $G/G^0$ are monogenic. We will use the following:

\begin{fact}\cite[Theo 1.8]{JC} \label{fact:disconnected}
For any definable group $G$ there are finite subgroups $F$ such that $G = FG^0$.
\end{fact}

\begin{prop} \label{prop:ab-conn}
Let $G$ be an abelian definable group. Then $G$ is monogenic if and only if $G^0$
and $G/G^0$ are monogenic.
\end{prop}

\begin{proof}
Suppose $G = \la g \df$ is monogenic. The quotient $G/G^0$ is monogenic by \Cref{lem:rk-dec}. In particular, any definable generator of $G$ projects to a definable generator of $G/G^0$. 

By \Cref{fact:disconnected}, there is a finite subgroup $F$ such that $G = FG^0$. Let $x \in F$, $y \in G^0$ such that $g = xy$,
and set $K = \la y \df$. We claim that $K = G^0$.

Denote by  $\ol y$ the image of $y$ and $g$ in $G/F$. Then $G/F = \la \ol y \df$, otherwise the pre-image of $\la \ol y \df$ would be a proper definable subgroup of $G$ containing $g$, contradiction. It follows that $G = FK$ and $K$ has finite index in $G$. Therefore $K$
contains $G^0$ and since $y \in G^0$, it must be $K = G^0$, as claimed.

Conversely, suppose $G/G^0 = \la \ol x \df$ and $G^0 = \la y \df$. Let $g = xy$, where $x \in F$ is an element in the pre-image of $\ol x$, and set $K = \la g \df$.
We claim that $G = K$.

The image of $g$ in $G/F$ through the canonical projection is equal to image of $y$, so it is a definable generator, and $G = FK$. Hence $\dim K = \dim G$ and $G^0 \subseteq K$.

On the other hand, $g$ is mapped to the definable generator $\ol x$ through the canonical projection $G \to G/G^0$, so $G = KG^0$. Therefore, $G = K$.  
\end{proof}

\begin{cor}\label{cor:monogenic}
	A definable group $G$ is monogenic if and only if $G$ is a direct product of a monogenic definably connected group and a finite cyclic group.
\end{cor}

\begin{proof}
	 Suppose $G$ is monogenic. By \Cref{prop:ab-conn}, $G^0$ and $G/G^0$ are both monogenic. Since $G/G^0$ is finite, it must be cyclic. Finally, as $G^0$ is divisible, it splits in the abelian $G$ (see, for instance, \cite[10.24]{Rotman}) and $G \cong G^0 \times (G/G^0)$. The converse follows directly from \Cref{prop:ab-conn}.
	 \end{proof}

Similarly, for monogenic subgroups of a definable group, we can deduce the following:

\begin{cor} \label{cor:monogenic-subgroups}
	For every element $g$ of a definable group $G$, $\la g \df$ is a direct product of $\la g \df^0$ and a finite cyclic group. Moreover, $\la g \df^0$ is monogenic as well.
\end{cor}

In this and the following sections we will make occasionally use of the \textbf{o-minimal Euler characteristic} $E(G)$. We refer to Section 2 of  \cite{JC} for some background and main properties. We just recall here from \cite{Strzebonski} that a definable group $G$ is \textbf{$0$-group} when $E(G/H) = 0$ for any proper definable subgroup $H$. Any definably connected group is either torsion-free or contains an infinite $0$-group. Maximal $0$-groups exist in any definable group, they are all conjugate and called \textbf{$0$-Sylow subgroups}.

\medskip
We collect below some results about definable groups we will be using throughout the paper:

\begin{fact}\label{fact:fact}
	Let $G$ be a definably connected group.
	
	\begin{enumerate}[(a)]
		
		\medskip
		\item $G$ contains a unique maximal definably connected normal solvable subgroup $R$, and the quotient $G/R$ is semisimple.  
		
		\medskip
		\item The quotient $G/Z(G)$ is a linear group. In particular, if $G$ is semisimple, then $Z(G)$ is finite and $G/Z(G)$ is a centerless linear definable group. \cite[Theo 5.17]{Frecon}, \cite[Cor 3.3]{OPP96} 
		
		\medskip
		\item If $G$ is linear, then its $0$-Sylow $S$ is definably compact and a (maximal) definable torus. That is, $S$ is abelian, definably connected and definably compact. \cite[Lem 3.9]{PPSIII}, \cite[Lem 2.7]{me2}
		
		\medskip
		\item If $G$ is a linear definable torus, then $G$ is definably isomorphic to $\SO_2(\Rs)^{\dim G}$ for some definable real closed field $\Rs$. \cite[\S 3]{PPSIII}
		
		\medskip
		\item If $G$ is not definably compact, then $G$ contains a definable $1$-dimensional torsion-free subgroup. \cite[Theo 1.2]{PS}
		
		\medskip
		\item $G$ contains a maximal normal definable torsion-free definable subgroup denoted by $\mtf(G)$. If $G$ is solvable, then $G = \mtf(G)A$ for any $0$-Sylow subgroup $A$.\cite[Prop 2.1]{CPI}, \cite[Prop 3.1]{me2}
		
		\medskip
		\item If $\mtf(G)$ is trivial, then $G$ is a central extension of a semisimple group and $G = Z(G)[G, G]$.  \cite[Theo 1.4]{JC}, \cite{HPP}
		\end{enumerate}
	
\end{fact}

\medskip
Recall that a group $G$ is called {\bf periodic}, or a torsion group, when every element in $G$ has finite order. Countable abelian divisible periodic groups are uniquely determined by their torsion subgroups. For reader's convenience, we include a proof of the following extracted from \cite{Fuchs} and \cite{Kaplansky}.

                    
\begin{fact}\label{fact:periodic}
Let $H$ be a countable abelian divisible periodic group. If there is a positive $n \in \N$ such that the $k$-torsion subgroup of $H$ is isomorphic to $(\Z/k\Z)^n$ for each $k \in \N$, then $H$ is isomorphic to $(\Q/\Z)^n$.  
\end{fact}   

\begin{proof}
Since $H$ is a periodic abelian group, it decomposes into the direct sum of its $p$-primary components:
\[ H \cong \bigoplus_{p \in \mathbb{P}} H_p, \]
where $\mathbb{P}$ is the set of prime numbers and $H_p = \{x \in H \mid p^m x = 0 \text{ for some } m \ge 0\}$.

Similarly,  
\[ (\mathbb{Q}/\mathbb{Z})^n \cong \left( \bigoplus_{p \in \mathbb{P}} \mathbb{Z}(p^\infty) \right)^n \cong \bigoplus_{p \in \mathbb{P}} (\mathbb{Z}(p^\infty))^n. \]

It suffices to show that for each prime $p$, the component $H_p$ is isomorphic to $(\mathbb{Z}(p^\infty))^n$.

Fix a prime $p$. The condition $H[k] \cong (\mathbb{Z}/k\mathbb{Z})^n$ holds for all $k$, so specifically for prime powers $k = p^m$ ($m \ge 1$).
Note that for any periodic group, the $p^m$-torsion elements reside entirely within the $p$-primary component $H_p$. Thus:
\[ H_p[p^m] = H[p^m] \cong (\mathbb{Z}/p^m\mathbb{Z})^n. \]
In particular, the cardinalities are:
\[ |H_p[p^m]| = (p^m)^n = p^{mn}. \]

We claim that $H_p$ is a divisible group.  
Let $m \ge 1$. Consider the homomorphism $\phi_m: H_p[p^{m+1}] \to H_p[p^m]$ defined by multiplication by $p$. Its kernel is
\[ \ker(\phi_m) = \{ x \in H_p[p^{m+1}] \mid px = 0 \} = H_p[p]. \]
We know $H_p[p] \cong (\mathbb{Z}/p\mathbb{Z})^n$, so $|\ker(\phi_m)| = p^n$. Therefore, by cardinality reasons, the map $\phi_m$ is surjective
 and for every $y \in H_p[p^m]$, there exists $x \in H_p[p^{m+1}]$ such that $px = y$.

Since $H_p = \bigcup_{m=1}^\infty H_p[p^m]$, for any $y \in H_p$, $y \in H_p[p^m]$ for some $m$. The surjectivity of $\phi_m$ guarantees $y$ is divisible by $p$ within $H_p$. Therefore, $H_p$ is divisible.

Being $H_p$ a divisible torsion group, it is a direct sum of copies of $\mathbb{Z}(p^\infty)$:
\[ H_p \cong \bigoplus_{i \in I} \mathbb{Z}(p^\infty). \]
 By our hypothesis, $H_p[p] \cong (\mathbb{Z}/p\mathbb{Z})^n$. Since these are vector spaces over the field $\mathbb{F}_p$, their dimensions must match, so $|I| = n$ and we have
 \[ H \cong \bigoplus_{p} (\mathbb{Z}(p^\infty))^n \cong \left( \bigoplus_p \mathbb{Z}(p^\infty) \right)^n \cong (\mathbb{Q}/\mathbb{Z})^n. \]
\end{proof}

\begin{prop}\label{prop:periodic}
An $n$-dimensional definable group $G$ is periodic if and only if $G^0 \cong (\Q/\Z)^n$.
\end{prop}

\begin{proof}
Suppose $G$ is a periodic $n$-dimensional definable group. If $n = 0$, there is nothing to prove. So assume $n \geq 1$. By \Cref{fact:fact}(e), $G$ is definably compact.  

If $G^0$ is not solvable, set $H$ to be the quotient of $G^0$ by its solvable radical. Since $H$ is semisimple, its center is finite and $H/Z(H)$ is a linear definably compact group. Let $T$ be its maximal definable torus. Then $T$ is definably isomorphic to $\SO_2(\Rs)^{\dim T}$, for some definable real closed field $\Rs$ (see \Cref{fact:fact}). 

Note that the field of real algebraic numbers $\R_{alg}$ embeds in any real closed field, hence $\SO_2(\R_{alg}) \subseteq \SO_2(\Rs)$. However, 
$\SO_2(\R_{alg})$ is not periodic, so $H$ is not periodic and $G$ is not periodic, contradiction. It follows that $G^0$ is solvable.

By \cite[Corollary 5.4]{Peterzil-Starchenko1}, $G^0$ is abelian and by \cite[Theorem 1.1]{EMPRT}, its $k$-torsion subgroup $G^0[k]$ is isomorphic to 
$(\Z/k\Z)^n$. As $G^0$ is periodic, $G^0 = \bigcup_{k \in \N} G^0[k]$ is countable and it is isomorphic to $(\Q/\Z)^n$ by \Cref{fact:periodic}.

 Conversely, assume $G^0 \cong (\Q/\Z)^n$. We need to check that if $G$ is not definably connected, $G$ is still periodic. By \Cref{fact:disconnected}, there is a finite subgroup $F$ such that $G = FG^0$. Since $G^0$ is normal, if $|F| = p$, then $x^p \in G^0$ for any $x \in G$.
 Therefore, if $G^0$ is periodic, then $G$ is periodic. 
\end{proof}

\begin{lem}\label{lem:tf}
If $G$ is an infinite torsion-free definable group, then 

\[
\rk G \leq \dim G.
\]
\end{lem}

\begin{proof}
By induction on $n = \dim G$. If $n = 1$, then $G$ has no proper non-trivial definable subgroup, so $G = \la g \df$ for any $g \neq e$.

Assume $n > 1$. By \cite[Cor 2.12]{PeSta}, $G$ has a normal definable subgroup $H$ of co-dimension $1$. By induction hypothesis, $\rk(H) \leq n-1$ and $\rk(G) \leq \rk(H) + \rk(G/H) \leq n$ by \Cref{lem:rk-dec}.
\end{proof}

\begin{rem}
	As noted by the referee, the paper \cite{PeSta} where it was shown that every torsion-free definable group is solvable and it has a normal definable subgroup of co-dimension 1 is set in o-minimal expansions of a field, while we are using it in arbitrary o-minimal structures. To convince ourselves that no assumption on the structure is being made, note that by \cite[Theorem 5.1]{PPSIII}, any definably simple group $G$ in an arbitrary o-minimal structure is elementarily equivalent to a simple Lie group (which contains an infinite torus), and therefore $E(G) = 0$. However, a definable group $G$ is torsion-free if and only if $|E(G)| = 1$ by \cite{Strzebonski}, so $G$ cannot have any infinite semisimple definable quotient and it must be solvable.
	
	Once we know that $G$ is solvable and $A$ is a normal abelian subgroup, then $\la A \df$ is an abelian normal definable subgroup that must be infinite because $G$ is torsion-free. By an induction argument we find a normal definable subgroup of co-dimension 1 in $G/\la A \df$, and we are then reduced to the case where $G$ is abelian. In this case, the existence of a $1$-dimensional torsion-free definable subgroup \Cref{fact:fact}(e) suffices, as noted in \cite[Fact 2.7]{PeSta}.  
\end{rem}

The inequality from \Cref{lem:tf} may or not be strict. Below are easy examples of both cases. We will show in the proof of \Cref{theo:def-fg} that the inequality holds for any definably connected definably finitely generated group. Moreover, for a definable group $G$ that is not solvable (unlike torsion-free groups) the inequality is always strict.

\begin{exs}
Let $\M$ be the real field. If $G = (\R, +)^2$, then $\rk G = 2$, since for any non-zero 
$x \in G$, $\la x \df$ is the line through $x$ and the origin. On the other hand, if 
$G = (\R, +) \times (\R^{>0}, \cdot)$, then $G$ is monogenic. For instance, let $H = \la (1, 2) \df$. If $\dim H = 1$, then $H$ is definably isomorphic to both $(\R, +)$ and 
$(\R^{>0}, \cdot)$, that are not definably isomorphic in \M, contradiction. Therefore $H=G$.
\end{exs}

\begin{proof}[Proof of \Cref{theo:linear}]
Let $G$ be a $n$-dimensional definable subgroup of $\GL_m(\Rs)$. We will prove our claim by induction on $n$. Assume first $G$ is definably connected.

If $n = 1$, then either $G$ is torsion-free or $G$ is definably isomorphic to $\SO_2(\Rs)$ (see \Cref{fact:fact}). In either case, $G$ is not periodic as noted in the proof of \Cref{prop:periodic}. It follows that $G = \la g \df$ for any non-torsion element $g \in G$.  

Assume $n > 1$. If $G$ is periodic, then by \Cref{prop:periodic}, $G$ is isomorphic to $(\Q/\Z)^n$. On the other hand, by \Cref{fact:fact}, $G$ is definably isomorphic to $\SO_2(\Rs)^n$ that is not periodic, contradiction. It follows that $G$ is not periodic.

Assume $G$ is \emph{definably simple}. That is, $G$ is not abelian and contains no proper non-trivial normal definable subgroup.
Let $g \in G$ be a non-torsion element and set $A = \la g \df$. Note that $A^0$ is monogenic too by \Cref{cor:monogenic-subgroups}. Because $G$ is not abelian, $A^0 \neq G$, and let  
$H$ be a proper definably connected subgroup of $G$ containing $A^0$ of largest possible dimension. Since $G$ is definably connected, $\dim H < \dim G$ and $H$ has finite definable rank by induction hypothesis. Because $G$ is definably simple, $H$ is not normal and there is $x \in G$ such that $H \neq H^x$, where $H^x$ denotes the conjugate of $H$ by $x$. Define $K = \la H \cup \{x\} \df$.

If $K^0 = G$, then $K = G$ by connectedness, and $\rk(G) \leq \rk(H) + 1$. Suppose $K^0 \neq G$.  
Since $H$ is definably connected, $H \se K^0$. By maximality of $H$, it must be $H = K^0$.
On the other hand, $H^x$ is definably connected too, so $H^x \se K^0$. Moreover, because $\dim H = \dim (H^x)$, it must be 
$H^x = K^0 = H$, contradiction. Therefore, $K^0 = G$ and $G$ has finite definable rank.

If $G$ is \emph{semisimple}, then $Z(G)$ is finite and by \cite{PPSI} $G/Z(G)$ is a direct product of finitely many definably simple groups, which are all linear (\Cref{fact:fact}). Therefore $G$ has finite definable rank by the simple case.

If $G$ is \emph{solvable}, then by \Cref{fact:fact} $G = \mtf(G)A$, where $\mtf(G)$ is torsion-free and $A$ is definably isomorphic to $\SO_2(\Rs)^{\dim A}$ for some definable real closed field $\Rs$.
Since both $\mtf(G)$ and $\SO_2(\Rs)$ have finite definable rank (\Cref{lem:tf}), so does $G$.  

If $G$ is not solvable nor semisimple, then $G$ has finite definable rank by the solvable and semisimple cases (see \Cref{fact:fact}(a) and \Cref{lem:rk-dec}).

Finally, if $G$ is not definably connected, then $\rk(G) \leq \rk(G^0) + |G/G^0|$ is finite by \Cref{cor:disc} and the connected case.
\end{proof}

\begin{cor}\label{cor:center}
If the center of a definably connected group $G$ is definably finitely generated, then $G$ is definably finitely generated.
\end{cor}

\begin{proof}
	 By \Cref{fact:fact}(b), \Cref{theo:linear} and \Cref{lem:rk-dec}.
\end{proof}

By \Cref{fact:fact}(f), the definably connected component of the center $Z(G)^0$ can be decomposed as a product of a torsion-free definable group and its $0$-Sylow. Since torsion-free definable groups are definably finitely generated, we obtain:

\begin{cor}
	If a definable group $G$ has infinite definable rank then the $0$-Sylow of $Z(G^0)$ has infinite definable rank. 
	\end{cor}

\section{A structure theorem}

This section is devoted to the proof of our main result \Cref{theo:def-fg}.

 \begin{proof}[Proof of \Cref{theo:def-fg}]
 Let $H$ be a definably connected definably finitely generated subgroup of $G$ of largest possible dimension. Suppose $K$ is an arbitrary definably connected subgroup of $G$ with finite definable rank and set $A = \la H \cup K\df$.  Since $K$ is definably connected, $K \se A^0$. On the other hand, $\dim A^0 \leq \dim H$ by maximality of $H$, so $A^0 = H$ and $K \se H$. Therefore $H$ is unique, it contains all definably connected subgroups of $G$ with finite definable rank, and $H = \fg$.

 If $f$ is a definable automorphism of $G$, then its image $f(\fg)$ is a definably connected subgroup with the same dimension and definable rank as $\fg$. Hence $\fg$ is definably characteristic and normal in $G$.

 \begin{claim}\label{claim:fg-periodic}
 	$G/\fg$ is a periodic group.
 \end{claim}
 
 \begin{proof}
 By \Cref{claim:gzero} and maximality of $\fg$, 
$\la \fg \cup \{x\} \df^0 = \fg$ for each $x \in G$. By \Cref{fact:disconnected}, $\la \fg \cup \{x\} \df = F\fg$ for some finite subgroup $F$. In particular, for each $x \in G$, $x = a g$ where $a \in F$ has finite order and $g \in \fg$. Since for each $n \in \N$, $x^n = a^ng'$ for some $g' \in \fg$, it follows that 
$x^{|F|} = g' \in \fg$ and $G/\fg$ is a periodic group, as claimed.
\end{proof} 

\begin{claim}\label{claim:H-completes}
$\fg$ maps surjectively onto every definably finitely generated definable quotient of $G^0$.
\end{claim}

\begin{proof}
Suppose $G = G^0$ is definably connected and $N$ is a normal definable subgroup of $G$ such that $G/N$ has finite definable rank. We want to show that $G = N\fg$.

Note that $N\fg$ is a normal definable subgroup of $G$, because so are $N$ and $\fg$.
On the one hand, 
\[
\dfrac{G}{N\fg} \cong \dfrac{G/N}{(N\fg)/N} 
\]

is definably finitely generated because every definable quotient of a definably generated group is definably generated (\Cref{lem:rk-dec}). On the other hand,
\[
\dfrac{G}{N\fg} \cong \dfrac{G/\fg}{(N\fg)/\fg} 
\]

is a definably connected periodic group, because so is $G/\fg$. By \Cref{prop:periodic}, the only definably connected definably finitely generated periodic group is the trivial one, and 
$G = N\fg$, as claimed.
\end{proof}

\begin{claim}\label{claim:derived}
	$G^0/\fg$ is abelian.
\end{claim}

\begin{proof}
	This follows directly from Claim 3.1 and \Cref{prop:periodic}, but we give below another proof, that does not assume $G/\fg$ is periodic. Indeed, we can show by induction on $n = \dim G$ that for every definably connected group $G$, the definable subgroup $D = \la [G, G] \df$ generated by the commutator subgroup of $G$ has always finite definable rank.

	  If $n = 1$, $G$ is abelian and there is nothing to prove. Assume $n > 1$.
	If $\mtf(G)$ is not trivial, set $\bar G = G/\mtf(G)$. By induction hypothesis, $\la [\bar G, \bar G] \df$ has finite definable rank. Since $D$ maps surjectively onto $\la [\bar G, \bar G] \df$ through the canonical homomorphism $G \to \bar G$, $D$ is an extension of a definably finitely generated group by a torsion-free definable group. Hence $D$ is definably finitely generated by \Cref{lem:rk-dec} and \Cref{lem:tf}. So assume $\mtf(G)$ is trivial. By \Cref{fact:fact}(g), $G$ is a central extension of a semisimple group and $G = Z(G) [G, G]$.   
	
	The semisimple definable group $G/Z(G)$ is definably finitely generated by \Cref{theo:linear} and \Cref{fact:fact}. Given $X$ a finite set of definable generators for $G/Z(G)$, let $H$ be the definable subgroup definably generated by a set of pre-images of $X$ in $G$. Then $H$ maps surjectively onto $G/Z(G)$ and 
	$G = Z(G)H$. It follows that $[G, G] = [H, H]$ and $D \se H$ by minimality of $D$.
	If $D = G$, then $H = G$ is definably finitely generated. If $D \neq G$, then $[D, D] = [G, G]$ nevertheless, since $D$ is a central extension of $[G, G]$. By induction hypothesis, $ \la [D, D] \df$ is definably finitely generated, and we are done. 
\end{proof}

\medskip
To show the claimed decomposition of $G^0$ as an almost direct product, assume $G$ is definably connected. Let $p = \dim (G/\fg)$. By \Cref{prop:periodic}, $G/\fg = (\Q/\Z)^p$. Since $G/Z(G)$ is a linear group (\Cref{fact:fact}(b)), we know that it is definably generated by \Cref{theo:linear}. For ease of notation, set $Z = Z(G)^0$. The quotient $G/Z$ is an extension of the definaby finitely generated $G/Z(G)$ by the finite group $Z(G)/Z$, hence it is definably finitely generated as well by \Cref{lem:rk-dec}. It follows from \Cref{claim:H-completes} that $G = Z\fg$.

Set $K = (Z \cap \fg)^0$. Note that
\[
G/\fg = (Z \fg)/\fg \cong Z/(Z \cap \fg) \cong \dfrac{Z/K}{(Z \cap \fg)/K} \cong (\Q/\Z)^p
\]

where $K$ has finite index in $Z \cap \fg$. It follows that $Z/K$ is periodic as well and isomorphic to $(\Q/\Z)^p$ too.

Because $K$ is definably connected, it is divisible and it splits in the abelian group $Z$ (see, for instance, \cite[10.24]{Rotman}). That is, there is a central subgroup $A$ isomorphic to $(\Q/\Z)^p$ such that
\[
Z = K \times A.
\]

Since $G = Z \fg$ and $K \se \fg$, it follows that $G = \fg A$.

Moreover, because $K \cap A$ is trivial, $\fg \cap A$ embeds in $(\fg \cap Z)/K$, which is a finite group by definition of $K$. Hence $\fg \cap A$ is finite. 

Let $F = \fg \cap A \subset Z$. Then the claimed decomposition  
\[
 G^0 \cong (\fg \times  (\Q/\Z)^p)/F
\]
 is proved.

\bigskip
We now move to compare the definable rank of a definably finitely generated group with its dimension.

\begin{claim}\label{claim:rk-dimension}
 \[
\rk(\fg) \leq \dim \fg.
\]
\end{claim}

\begin{proof}
For ease of notation, assume $G$ is an arbitrary definably connected definably finitely generated group. We will show that $\rk(G) \leq \dim G$ by induction on $n = \dim G$. First note that by \Cref{prop:periodic} no definably connected periodic group has finite definable rank. Hence $G$ is not periodic. Let $g \in G$ be a non-torsion element.

If $n = 1$, then $G = \la g \df$. Suppose $n > 1$. The infinite 
$\la g \df^0$ is monogenic as well by \Cref{prop:ab-conn} and \Cref{cor:monogenic-subgroups}. Let $H$ be a proper definably connected definably finitely generated subgroup of largest possible dimension containing $\la g \df^0$ (If no such $H$ exists, then $G$ is monogenic, and there is nothing to prove). By induction hypothesis, $\rk(H) \leq \dim H$.

If $H$ is normal, then $\rk (G/H) \leq \dim (G/H)$ by induction hypothesis as well, and $\rk(G) \leq \dim G$ by \Cref{lem:rk-dec}.

If $H$ is not normal, let $x \in G$ such that $H^x \neq H$ and set 
$K = \la H \cup \{x\} \df$. If $K \neq G$, then $K^0 = H$ by \Cref{claim:gzero} and maximality of $H$. On the other hand, $H^x$ is definably connected, so $H^x \se K^0$. Since $\dim H^x = \dim H$, it must be $H^x = K^0 = H$, contradiction. So $K = G$ and $\rk G \leq \dim H + 1 \leq \dim G$, as claimed.
\end{proof}

Finally, we need to show that the equality can hold only if $\fg$ is solvable. To this end, we will first show a few intermediate results of independent interest.

\begin{claim}\label{claim:1dim-tf-central}
	If $\dim \mtf(G) = 1$, then $\mtf(G)$ is central in $G^0$.
\end{claim}

\begin{proof}
	For ease of notation, set $N = \mtf(G)$ and assume $G$ is definably connected. As showed in \cite[Lemma 4.4]{me-nilpotent}, $\Aut N$ is a definable 1-dimensional group and $(\Aut N)^0$ is torsion-free.

	Let $s \colon G/N \to G$ be a definable section of the canonical projection $\pi \colon G \to G/N$. Then there is a unique ordered pair $(a, x) \in N \times G/N$ such that $g = s(x)a$. For every $g \in G$, the conjugation map $f_g \colon N \to N$ mapping $a \mapsto gag^{-1}$ is a definable automorphism of $N$, so there is a homomorphism $\Phi \colon G \to \Aut(N)$, given by $g \mapsto f_g$ such that $N \subseteq \ker \Phi$, being $N$ abelian. Thus $\Phi $ induces the definable homomorphism
	\begin{align*}
		\varphi \colon & G/N \longrightarrow \Aut (N) \\
		& \ \  x \mapsto (a \mapsto s(x)a s(x)^{-1})  
	\end{align*}
	
	\vs \noindent
	which does not depend on the choice of the section $s$.
	We claim that $\varphi(G/N) = \{e\}$, and so $N \subseteq Z(G)$. If not, 
	$\varphi(G/N)$ is an a 1-dimensional definable torsion-free group, since $G$ is definably connected. However, if $R$ is the solvable radical of $G$, then $R/N$ is definably compact (\cite[Theorem 3.1]{diagram}) and $G/R$ is perfect (\cite[Claim 3.1]{HPP}). So $\varphi(G/N)$ is trivial and $\mtf(G) \subseteq Z(G)$.
\end{proof}

\begin{claim}\label{claim:semisimple}
	If $G^0$ is not solvable, then $G$ contains a definably connected non-abelian subgroup $H$ such that $\rk H = 2 < \dim H$.
\end{claim}

\begin{proof}
 For ease of notation, let $G$ be definably connected.
   Suppose first $G$ is semisimple. By \Cref{fact:fact}(b), we can assume $G$ is a linear group. Say $G$ definably embeds in $\GL_m(\Rs)$. Let $T$ be a $0$-Sylow subgroup.  By \Cref{fact:fact}(c)-(d) $T$ is definably isomorphic to $\SO_2(\Rs)^{\dim T}$  and $T$ is monogenic.

 Because $G$ is semisimple, $T$ cannot be normal. Let $x \in G$ such that $T \neq T^x$ and set $H = \la T \cup T^x\df$. Since $T$ and $T^x$ are definably connected, they are contained in $H^0$, hence $H$ is definably connected. If $H$ is abelian, then it contains a unique $0$-Sylow subgroup and $T = T^x$, contradiction. Therefore, $H$ is not abelian
 and $\rk H > 1$. Since $T$ is monogenic, $\rk H = 2$.

   Suppose, by way of contradiction, $\dim H = 2$. Thus $H$ is solvable. $H$ cannot be definably compact, because solvable definably connected definably compact groups are abelian by \cite{Peterzil-Starchenko1}. Hence $\dim \mtf(H) = 1$ and by \Cref{claim:1dim-tf-central} $\mtf(H)$ is central. This implies $H$ is abelian, in contradiction with $T \neq T^x$. Therefore $\dim H > 2 = \rk(H)$.
 
 Suppose now $G$ is not semisimple and let $R$ be its solvable radical. Because $G/R$ is semisimple, there is a non-abelian definably connected subgroup $\ol H = \la a, b\df$ in $G/R$ such that $\dim \ol H > 2$. Let $H$ be of minimal dimension among the 
 definable subgroups of $G$ satisfying $(HR)/R=\ol H$. Then 
 $(H^0 R)/R= \ol{H}$, so we may assume that $H$ is definably connected. 
 Moreover, minimality gives $H= \la x,y \df$, for any preimages $x,y\in H$ of $a$ and $b$, and $\dim H \geq \dim \ol H > 2$. The definable generators $x$ and $y$ do not commute, because their images $a$ and $b$ do not commute, so $H$ is not abelian.
 \end{proof}

\begin{claim}\label{claim:subgr-enough}
	If $H$ is a definably connected subgroup of $\fg$ such that $\rk (H) < \dim H$, then $\rk(\fg) < \dim \fg$.	 
\end{claim}

\begin{proof}
	Assume $G$ is a definably connected definably finitely generated group and $H$ a definably connected subgroup such that $\rk (H) < \dim H$. We will prove our claim by induction on $n = \dim G$. If $n = 1$, there is nothing to prove, so assume $n > 1$. Let $A$ be a proper definably connected definably finitely generated subgroup of largest possible dimension containing $H$. By induction hypothesis, $\rk (A) < \dim A$.
	
	If $A$ is normal, then $\rk (G/A) \leq \dim (G/A)$ by \Cref{lem:rk-dec} and \Cref{claim:rk-dimension}. Since $\dim G = \dim A + \dim(G/A)$, it follows by \Cref{lem:rk-dec}  that $\rk(G) < \dim G$, as claimed.
	
	If $A$ is not normal, let $x \in G$ such that $A^x \neq A$ and set 
	$K = \la A \cup \{x\}\df$. If $K \neq G$, then $K^0 = A$ by \Cref{claim:gzero} and maximality of $A$. It follows that $K^0 = A = A^x$, contradiction. Therefore, it must be $K = G$ and $\rk G \leq \rk A + 1 < \dim A + 1 \leq \dim G$.
	 \end{proof}

If $\fg$ is not solvable, we can conclude that $\rk(\fg) < \dim \fg$ by 
 \Cref{claim:semisimple} and
\Cref{claim:subgr-enough}, and  
the proof of \Cref{theo:def-fg} is complete.
\end{proof}

\begin{rem}
	Note that when $G$ is a definably connected $3$-dimensional semisimple group, the proof of Claim 3.6 provides an explicit definably generating set  $\{a, b\}$.  
\end{rem}

\begin{proof}[Proof of \Cref{cor:countable}]
	By \Cref{theo:def-fg} and \Cref{prop:periodic}, $\fg$ is definably finitely generated and $G^0/\fg$ is countable. \Cref{lem:rk-dec} applies with $H = \fg$ and $H = G^0$.
\end{proof}

\section{Structures with no countable interval}

Strzebonski showed in \cite[Cor 5.14]{Strzebonski} that semialgebraic $0$-groups 
in real closed fields with no countable interval are monogenic. With a different strategy, we generalize below his result to any $0$-group definable in an o-minimal structure with no countable interval. Unless otherwise stated, \textbf{in this section assume the ambient structure \M\ is an o-minimal structure with no countable interval}.

\begin{proof}[Proof of \Cref{theo:str}]
	Let $G$ be a $0$-group. We want to show that $G$ is monogenic. Suppose first $G$ is definably compact. Then $G$ and each of its definably connected subgroups is a definable torus and, in particular, a $0$-group. We claim that $G$ has countably many definable subgroups. 
	
	Indeed, $G$ has finitely many $k$-torsion elements for each $k \in \N$, so its torsion subgroup is countable. Moreover, every $0$-group is the definable subgroup generated by its torsion elements \cite[Prop 2.20]{me-nilpotent} and the set of $0$-subgroups of $G$ embeds in the set of finite-rank subgroups of $\Z^{\dim G}$, which is a countable set. It follows that $G$ has countably many definably connected subgroups.

	If $H$ is a definable subgroup that is not definably connected, it is an extension of a finite group by $H^0$. As $G/H^0$ has countably many finite subgroups, and there are at most countably many possibilities for $H^0$, $G$ has countably many definable subgroups, as claimed.
	
	Since $G$ is definably connected, each of its proper definable subgroup has smaller dimension and by \cite[Lemma 5.12]{Strzebonski}, no definable group can be covered by countably many definable subgroups of smaller dimension. Therefore, $G$ must be monogenic.
	
	Suppose now $G$ is not definably compact, and set $N = \mtf(G)$. The quotient
	$G/N$ is a definable torus, so monogenic by the above. Let $g \in G$ such that $\ol g \in G/N$ is a definable generator and set $K = \la g \df$. As $G = NK$,  
	\[
	E(G/K) = E(N/(N \cap K)) = \pm 1.
	\]
	
	\medskip \noindent
	However, $G$ is a $0$-group, so it must be $G = K$.
	
	Suppose now $G$ is an infinite periodic group. By \Cref{prop:periodic}, $G^0$ is a finite power of $\Q/\Z$. On the other hand, $G^0$ is abelian, definably connected and definably compact, hence a $0$-group. By the above, $G^0$ is trivial and $G$ is finite, as claimed.  
	
	If now $G$ is an arbitrary definable group, $G/\fg$ is a periodic group by \Cref{theo:def-fg}, hence it must be finite. Therefore $G^0 = \fg$ and $G$ has finite definable rank.
	\end{proof}

One may wonder if $0$-groups are the only monogenic abelian definably connected groups that are not $1$-dimensional. The answer is negative. For instance, if $S$ is Strzebonski's group from \Cref{ex:subgr},  $G = \R \times S$ is definably generated by $(1, x)$, where $x$ is any definable generator of $S$. In this case, $G$ is not a $0$-group, as $E(G/S) = E(\R) = -1$. In general, we prove below that whenever the maximal torsion-free definable subgroup $\mtf(G)$ of a definably connected abelian group $G$ is monogenic, then $G$ is monogenic as well. It will be a consequence of more general results about solvable groups. 

Note that if $G$ is a solvable definably connected definably compact group, then it is a $0$-group and monogenic by \Cref{theo:str}. If $G$ is not definably compact, then its definable rank is bounded above by the dimension of its maximal normal definable torsion-free subgroup $\mtf(G)$:

 \begin{prop}\label{prop:solvable}
 	Let $G$ be a solvable definably connected group. If $G$ is not definably compact, then
 	\[
 	\rk(G) \leq \dim \mtf(G).
 	\]
 	
 \end{prop}

\begin{proof} 
 Set $N = \mtf(G)$. Suppose $\rk(N) = k$ and $\dim N = d$. As torsion-free definable groups are definably connected, $k \leq d$ by \Cref{theo:def-fg}.

Fix a definably generating set $X = \{x_1, \dots, x_k\} \subset N$. We will prove our claim by induction on $n = \dim G$.  

If $G$ is torsion-free, $G = N$ and there is nothing to prove. Otherwise, $G = AN$ by \Cref{fact:fact}(f), where $A$ is a $0$-Sylow subgroup of $G$. By \Cref{theo:str}, $A$ is monogenic. Say, $A = \la a \df$.

Set $Y = \{a, x_2, \dots, x_k\}$ and $H = \la Y \df$. The subgroup 
$\la x_2, \dots x_k \df < N$ is torsion-free, so definably connected. Thus $\la x_2, \dots x_k \df \subset H^0$. Similarly, $A = \la a \df \subset H^0$. It follows that $H = H^0$, $\la H \cup \{x_1\} \df = G$ and $\rk(G) \leq \rk(H) + 1$. If $H = G$, $\rk(H) = \rk(G) \leq k$ and our claim is proved.

Suppose $H \neq G$ and $H$ is definable compact. Then $H = A$ and $N$ is monogenic. If $\dim N = 1$, then by \Cref{claim:1dim-tf-central} 
$N = \la x_1 \df$ is central and $G$ is abelian. In this case, $G = \la ax_1 \df$ and $\rk (G) = \dim N$. If $\dim N > 1$, then $\rk (G) \leq 2 \leq \dim N$. 

If $H \neq G$ and $H$ is not definably compact, then by induction hypothesis,  $\rk(H) \leq \dim \mtf(H)$. Moreover, $\dim \mtf(H) < \dim \mtf(G)$, since $A \subset H$. Putting everything together, we have
\[
\rk(G) \leq \rk(H) + 1 \leq \dim \mtf(H) + 1 \leq \dim \mtf(G).
\]
\end{proof}

\begin{cor}
Let $G$ be a definably connected group. If $\rk(G) = \dim G > 1$, then $G$ is torsion-free.  
\end{cor}
 
 \begin{proof}
 	By \Cref{theo:def-fg} $G$ is solvable, by \Cref{theo:str} $G$ is not definably compact, and \Cref{prop:solvable} $G$ is torsion-free.
 	
 \end{proof}

\medskip
When $G$ has a unique $0$-Sylow subgroup, \Cref{prop:solvable} can be improved by replacing the dimension of $\mtf(G)$ with its definable rank:

\begin{prop} \label{prop:solv-unique}
Let $G$ be a solvable definable connected group that is not definably compact. If $G$ has a unique (possibly trivial) $0$-Sylow subgroup, then 
\[
\rk(G) \leq \rk(\mtf(G)).
\]
\end{prop}

\begin{proof}
Set $N = \mtf(G)$ and suppose $\rk(N) = k$. Fix a definably generating set $X = \{x_1, \dots, x_k\} \subset N$. If $G$ is torsion-free, $G = N$ and there is nothing to prove. Otherwise, $G = AN$, where $A$ is a $0$-Sylow subgroup of $G$. By \Cref{theo:str}, $A$ is monogenic. Say, $A = \la a \df$. Set $Y = \{ax_1, \dots, ax_k \}$
and $K = \la Y \df$. We will show that $K = G$.

Let $p_1 \colon G \to G/N$ be the canonical projection, set $p_1(a) = \bar{a}$ and $H_1 = \la \bar{a} \df$. The pre-image of $H_1$ in $G$ is a definable subgroup containing $N$ and $a$ generating $A$. Hence it cannot be a proper definable subgroup. Since $p_1(ax_i) = p_1(a) = \bar{a}$, the restriction of $p_1$ to $K$ is a surjective map, and $G = NK$. Even if $K$ is not a normal subgroup of $G$, there is a definable bijection between definable sets $G/K$ and $N/(N \cap K)$. Hence $|E(G/K)|=1$ and every $0$-Sylow subgroup of $K$ is a $0$-Sylow subgroup of $G$. Because $G$ has a unique $0$-Sylow subgroup by assumption,  $A \subseteq K$.
In particular, $a \in K$. It follows that $X \subset K$ too, as $x_i = a\inv (ax_i) \in K$,
and $K = G$, as wanted.
\end{proof}

\begin{cor}\label{cor:ab}
Let $G$ be a definably connected abelian group. If $\mtf(G)$ is monogenic, $G$
is monogenic. 
\end{cor}

\begin{proof}
If $G$ is definably compact, then $G$ is a $0$-group and \Cref{theo:str} applies. If $G$ is not definably compact, then \Cref{prop:solv-unique} applies.
\end{proof}

 \begin{rem}
The converse of  \Cref{cor:ab} does not hold. The semialgebraic group $G$ in \Cref{ex:subgr} is definably connected and monogenic, but $\mtf(G) = H$ is not. 
\end{rem}

\medskip
\section{On Wiegold's problem}

 If $G$ is a perfect definably connected group, then it is definably finitely generated by \Cref{theo:def-fg}. On the other hand, when $G$ is not definably connected, it may have infinite definable rank. For instance, one can define in $(\Q, <, +)$ an action of $A_5$ on the infinite periodic group $[0, 1)^4$ with addition modulo 1 such that the corresponding semidirect product $G=[0, 1)^4 \rtimes A_5$ is a perfect locally finite group. We show below that in any case $G$ is normally monogenic, giving a positive answer to Wiegold's problem in the o-minimal setting.

\begin{proof}[\textbf{Proof of \Cref{theo:Wiegold}}]
We first consider the case where $G$ is definably connected. Since $G$ is perfect, $G$ is not solvable. 

\medskip
Assume $G$ is semisimple. By \cite[Theo 4.1]{PPSI}, the quotient $\ol G = G/Z(G)$ is a direct product $H_1 \times \dots \times H_k$
of definably simple groups $H_i$. We can see that $\ol G$ is normally monogenic by induction on $k$. It suffices to consider the case $k = 2$. 

Let $x_1 \in H_1$ be a $2$-torsion element and $x_2 \in H_2$ be a $3$-torsion element.
We claim that $(x_1, x_2) \in H_1 \times H_2$ is a normal generator. Set $N = \la (x_1, x_2) \df^{\ol G}$.  

As $(x_1, x_2)^2 = (e, x_2^2) \in N$ and $x_2^2 \neq e$, $N \cap (\{e\} \times H_2)$ is a normal non-trivial definable subgroup of the definably simple $H_2$. Hence $\{e\} \times H_2 \subset N$. Similarly, $(x_1, x_2)^3 = (x_1, e) \in N$ and $N \cap (H_1 \times \{e\}) = H_1 \times \{e\} \subset N$. Therefore, $N = \ol G$.

Note that we could have deduced $(x_1, e) \in N$ by $(x_1, x_2), (e, x_2) \in N$, regardless of the fact that $x_2$ is a $3$-torsion element. However, when $k > 2$, we need the stronger assumption for the $x_i's$ to be $p_{i+1}$-torsion elements to ensure that $N = \ol G$. We know that each $H_i$ contains torsion elements of each prime order because definably simple groups are elementarily equivalent to centerless simple Lie groups \cite[Theo 5.1]{PPSIII}, which always contain a $1$-dimensional torus.

Let now $g \in G$ be any element in the pre-image of a normal generator of $\ol G$ and set $N = \la g \df ^G$. As $Z(G)N = G$ and $Z(G)$ is finite, $\dim N = \dim G$.
However, $G$ is definably connected, so $G = N$ and $G$ is normally monogenic.

\medskip
Assume $G$ is not semisimple. Then $G$ contains a maximal solvable definably connected subgroup $R$, its solvable radical, such that $G/R$ is semisimple. By the semisimple case, $G/R$ is normally monogenic. Let $g \in G$ be any element in the pre-image of a normal generator of $G/R$.

Suppose $H$ is a normal definable subgroup containing $g$. Since $RH = G$, the quotient $G/H \cong R/(R \cap H)$ is solvable. However, $G$ is perfect, so it must be $G = H$. Therefore, $G = \la g \df^G$.

\bigskip
Let's now consider the case where $G$ is not definably connected. 

\medskip
\textbf{Claim I.} There is  $x \in G$ such that $\langle x \df^G = H$ satisfies 
$G = HG^0$ and $H$ is of minimal dimension.

 \begin{proof}[Proof of Claim I]
We know that $G/G^0$ is a perfect finite group, so normally monogenic. Let $\ol x \in G/G^0$ be a normal generator. For any $x \in G$ in the pre-image of $\ol x$, every normal definable subgroup $H$ containing $x$ maps onto $G/G^0$ and $G = HG^0$. Choose $H$ of minimal dimension among lifts of all normal generators of $G/G^0$.  
\end{proof}

\medskip
The quotient $G/H = (HG^0)/H \cong G^0/(G^0 \cap H)$ is a perfect definably connected group, so normally monogenic by the connected case above. Let $\ol y \in G/H$ be a normal generator and $y \in G$ an element in the pre-image of $\ol y$. Since $G^0$ maps onto $G/H$, we can take $y \in G^0$.

Set $g = xy$ and $K = \langle g \df^G$. Recall that $x$ is any element in the pre-image of a normal generator $\ol x \in G/G^0$. We will show that $g$ is a normal generator. That is, $K = G$.

Since $y \in G^0$, the image of $g$ in $G/G^0$ is $\ol x$ that normally generates. Therefore $K$ maps onto $G/G^0$ and $KG^0 = G$. 

On the other hand, the image of $g$ in $G/H$ is $\ol y$ that normally generates too.
Therefore $K$ maps onto $G/H$ and $KH = G$. Hence
\[
G/H = (KH)/H \cong K/(K \cap H) \cong G^0/(G^0 \cap H).
\]

\bigskip
\textbf{Claim II.} $(K \cap H)K^0 = K$.

\begin{proof}[Proof of Claim II]
	Since
	$K^0(K\cap H)/(K\cap H)$ is a definable subgroup of finite index of the connected definable group
	$K/(K\cap H) \cong G/H \cong G^0/(G^0\cap H)$, we have $K^0(K\cap H)=K$.
\end{proof}

\bigskip
\textbf{Claim III.}  There is a finite subgroup $F$ of $K \cap H$ such that $F G^0 = G$. 

\begin{proof}[Proof of Claim III]
By \Cref{fact:disconnected}, there is a finite subgroup $F$ of $K \cap H$ such that $K \cap H = F(K \cap H)^0$. Since $(K \cap H)^0 \subseteq K^0$, it follows that $K = (K \cap H)K^0 = FK^0$. Moreover, $KG^0 = G$ and $K^0 \subseteq G^0$, hence $FG^0 = G$.
 \end{proof}

\noindent
Because $G/G^0$ is normally monogenic, there is some $a \in F$ such that $\langle a \df^G \cdot G^0 = G$. Set $N = \langle a \df^G $. Since $H$ is of minimal dimension,

\[
\dim H \leq \dim N \leq \dim (K \cap H),
\]

\medskip
because $a \in K \cap H$, which is a normal definable subgroup of $G$, as $K$ and $H$ are.
It follows that $\dim H = \dim (K \cap H)$ and $H^0 \subseteq K$. Since $G = KH$, it must be $\dim G = \dim K$. Therefore, $G^0 = K^0$. On the other hand, 
$G = KG^0$ so $G = K$, as we wanted.
\end{proof}

\begin{que}\label{que:normally-mon}
Are there other structures/theories, besides o-minimal ones, where every definable perfect group is normally monogenic? 
\end{que}

\medskip \noindent
 \textbf{Acknowledgements.} Thanks to Yves de Cornulier for useful references and remarks about algebraic groups.  Thanks to Nicolas Monod for introducing me to Wiegold's problem. Thanks to the referee for the many comments and useful suggestions, and for pointing out a mistake in the proof of Claim 3.7, previously stated without the connectedness assumption.

%
%
%
%
%
%
%
%

\end{document}